\documentclass[4paper,leqno,[11pt]{amsart}

\usepackage[top=1.5in,bottom=1.3in,left=1.3in,right=1.3in,marginparwidth=1.5cm]{geometry}

\usepackage[frenchb,english]{babel}
\usepackage[utf8]{inputenc}

\usepackage{amsmath,amsthm,amssymb,amsfonts}

\usepackage[backref=page]{hyperref}
\renewcommand*{\backref}[1]{}
\renewcommand*{\backrefalt}[4]{%
    \ifcase #1 (Not cited.)%
    \or        (Cited on page~#2.)%
    \else      (Cited on pages~#2.)%
    \fi}

\hypersetup{
	colorlinks   =true,
	citecolor    =[rgb]{0.3,0,0.3},
	linkcolor    =blue,
	urlcolor     =[rgb]{0.3,0,0.3}
}

\numberwithin{equation}{section}
\usepackage{tikz}
\usetikzlibrary{cd}

\newtheorem{lem}{Lemma}
\newtheorem{thm}{Theorem}
\newtheorem{prop}{Proposition}

\theoremstyle{definition}
\newtheorem{defn}{Definition}
\newtheorem{exmp}{Example}

\newtheorem*{thankyou}{Acknowledgements}

\theoremstyle{remark}
\newtheorem*{rem}{Remark}

\newtheorem*{claim}{Claim}
\newtheorem*{pfclaim}{Proof of the Claim}

\renewcommand{\P}{\mathbb{P}}
\newcommand{\Q}{\mathbb{Q}}
\newcommand{\Z}{\mathbb{Z}}

\renewcommand{\O}{\mathcal{O}}
\newcommand{\F}{\mathcal{F}}

\newcommand\td{{\rm{td}}}
\newcommand\ch{{\rm{ch}}}

\makeatletter
\let\@wraptoccontribs\wraptoccontribs
\makeatother

\begin{document}

\title{Riemann-Roch Polynomials of the known Hyperkähler Manifolds}
\author{Ángel David Ríos Ortiz}
\address{Sapienza Universita di Roma, Dipartimento di Matematica, Piazzale Aldo Moro 5, 00185 Roma}
\address{Universit\'e Paris-Saclay, CNRS, Laboratoire de Math\'ematiques d'Orsay, B\^at. 307, 91405 Orsay, France}
\email{angel-david.rios-ortiz@universite-paris-saclay.fr}
\date{}
\contrib[with an appendix by]{Yalong Cao and Chen Jiang}
\address{Kavli Institute for the Physics and Mathematics of the Universe (WPI),The University of Tokyo Institutes for Advanced Study, The University of Tokyo, Kashiwa, Chiba 277-8583, Japan}
\email{yalong.cao@ipmu.jp}
\address{Shanghai Center for Mathematical Sciences, Fudan University, 
Shanghai 200438, China}
\email{chenjiang@fudan.edu.cn}

\maketitle

\begin{abstract}
    We compute explicit formulas for the Euler characteristic of line bundles in the two exceptional examples of Hyperkähler Manifolds introduced by O'Grady. In an Appendix, Chen Jiang and Yalong Cao use our formulas to compute the Chern numbers of the example of O'Grady in dimension 10.
\end{abstract}
\selectlanguage{french}
\begin{abstract}
    (Polynômes de Riemann-Roch pour les variétés hyperkählériennes connues) Nous calculons des formules explicites pour la caractéristique d'Euler de fibrés en droites pour les deux exemples exceptionnels de variétés hyperkählériennes introduits par O'Grady. Dans un appendice, Chen Jiang et Yalong Cao utilisent nos formules pour calculer le nombre de Chern de l'exemple d'O'Grady en dimension 10.
\end{abstract}

\section{Introduction}

A compact Kähler manifold is called \emph{Hyperkähler} (HK) if it is simply connected and carries a holomorphic symplectic form that spans $H^{2,0}$. HK manifolds can be thought of as the higher dimensional analogues of $K3$ surfaces, and they constitute one of the three fundamental classes of varieties with vanishing first Chern class \cite{Beauville83}.

Although any two $K3$ surfaces are deformation equivalent, this fact no longer holds in higher dimensions. The first two series of examples of deformation types in each (necesarily even) dimension were described by Beauville \cite{Beauville83}; the first series, denoted by $K3^{[n]}$, is given by the Hilbert scheme of $n$ points in a K3 surface. The other one is a submanifold in the Hilbert scheme of $n$ points in an abelian surface. Generalizing the construction of a Kummer surface, this ($2n$-dimensional) deformation type is denoted by $\text{Kum}_n$.

Later, O'Grady introduced two new deformation types in dimensions $6$ and $10$  (\cite{OG99},\cite{OG03}), now denoted by OG6 and OG10 respectively. The construction of both exceptional examples is done by resolving a singular moduli space of sheaves on a K3 surface for OG10 and an abelian surface for OG6. In view of this analogies it is expected that the projective geometry of HK manifolds of $K3^{[5]}$-type (respectively $\text{Kum}_3$-type) should be related with that of OG10-type (respectively OG6-type).

The main result of this paper (cf. Theorem \ref{thm:RR}) gives, for the HK manifolds described by O'Grady, closed formulas that compute the Euler characteristic of any line bundle in terms of numerical polynomials that only depend in the Beauville-Bogomolov form --a canonical quadratic form in the second cohomology group of any HK. Surprisingly the formulas turn out to be exactly the same as those of the series described by Beauville.

In order to compute these polynomials we use two different methods. The first one exploits a recent description in \cite{LSV17} of OG10 as a compactification of a fibration associated with a cubic $4$-fold. The second one is based on the explicit descriptions of some uniruled divisors in two different models of OG6 given in \cite{MRS2018} and \cite{Nagai2014}.

Observe that in \cite{CaoChen19} the authors give a closed formula for the Riemann-Roch polynomial of OG6 in terms of the so-called $\lambda$-invariant, in our work the closed formula is obtained directly.

Finally we would like to point out the very recent paper \cite{Jiang20} where Chen Jiang proves the positivity of the coefficients of the Riemann-Roch polynomial for HK manifolds in general.

\begin{thankyou}
I am grateful to my PhD advisor Kieran O'Grady for his patience and deep insight. Thanks also goes to Federico Caucci, Antonio Rapagnetta and Domenico Fiorenza for many stimulating conversations and their mathematical suggestions. Finally I would like to thank heartily Chen Jiang and Yalong Cao for the interest they took in this work and for writing the Appendix. The author was supported by the European Research Council (ERC) under the European Union’s Horizon 2020 research and innovation programme (ERC-2020-SyG-854361-HyperK)
\end{thankyou}

\section{Preliminaries}

Let $X$ be a HK manifold of dimension $2n$ and $q_X$ its Beauville-Bogomolov form \cite{Beauville83}. Recall that the Fujiki constant $c_X$ is defined as the rational  number such that for all $\alpha\in H^2(X)$ we have the so-called Fujiki relation:
\begin{equation}\label{eq:fujiki}
    \int_X \alpha^{2n} = c_X q_X(\alpha)^n
\end{equation}
\begin{rem}
The polarized form of Fujiki's relation is
\begin{equation}
    \int_X \alpha_1\smile\dots\smile\alpha_{2n} = \frac{c_X}{(2n)!}\sum_{\sigma\in S_{2n}} q_X(\alpha_{\sigma(1)},\alpha_{\sigma(2)})\cdot \cdot \cdot q_X(\alpha_{\sigma(2n-1)},\alpha_{\sigma(2n)})
\end{equation}
\end{rem}

Huybrechts further generalized this relation to all polynomials in the Chern classes, more specifically he proved the following

\begin{thm}[\cite{GHJ2003}, Corollary 23.17]\label{thm:constants}
Assume $\alpha\in H^{4j}(X,\Q)$ is of type $(2k,2k)$ for all small deformations of $X$. Then there exists a constant $C(\alpha)\in\Q$ such that
\begin{equation}\label{eq:todd}
    \int_X \alpha \smile \beta^{2n-2k} = C(\alpha)\cdot q_X(\beta)^{n-k}
\end{equation}
for all $\beta\in H^2(X,\Q)$.
\end{thm}
\begin{rem}
If we set $\alpha = 1$ in Theorem \ref{thm:constants} we obtain the Fujiki relation \eqref{eq:fujiki} and also that $a_n = c_X$.
\end{rem}

The odd Chern classes (hence the odd Todd classes) of $X$ vanish since the symplectic form on $X$ induces an isomorphism between $T_X$ and its dual. The Todd classes are topological invariants of $X$, so for any line bundle $L$ in $X$ we combine Theorem \ref{thm:constants} with Hirzebruch-Riemann-Roch Theorem to get 
\begin{equation}\label{eq:chi}
   \chi(X,L) =  \sum_{i=0}^n\frac{1}{(2i)!}\int_X Td_{2n-2i}(X)\smile c_1(L)^{2i} = \sum_{i=0}^n \frac{a_i}{(2i)!} \cdot q_X(L)^{i}
\end{equation}
where $a_i := C(Td_{2n-2i}(X))$.
\begin{defn}[\emph{Huybrechts, Nieper-Wißkirchen, Riess}]
The Riemann-Roch polynomial of $X$, denoted by $RR_X(t)$, is the polynomial 
\[
RR_X(t) = \sum_{i=0}^n \frac{a_i}{(2i)!}t^i.
\]
\end{defn}

Let us list a few well-known properties of this polynomial.

\begin{lem}\label{lem:RRproperties}
Let $X$ be a HK variety of dimension $2n$. The following properties hold:
\begin{enumerate}
    \item $RR_X$ depends only on the deformation class of $X$.
    \item The constant term is $a_0 = n+1$.
    \item The coefficient of the highest-order term  is $a_n = c_X$ and is positive.
    \item The coefficient $a_{n-1}$ is positive.
\end{enumerate}
\end{lem}
\begin{proof}
We have already observed that the Todd classes are a deformation invariant of $X$. Hence each $a_i$ (and therefore $RR_X$) is also a deformation invariant of $X$. The constant term of $RR_X$ is the holomorphic Euler characteristic of $X$, this was computed \cite{Beauville83} to be $n+1$. The constant $a_n = C(Td_0(X))$ is given by \eqref{eq:fujiki} so it is equal to $c_X$. Observe that $c_X$ is positive because the left hand side of the Fujiki relation \eqref{eq:fujiki} is a volume form.

By the first item we can assume $X$ to be projective. Nieper \cite{Nieper03} computed
\[
\int_X c_2(X)\smile c_1(L)^{2n-2} = \binom{2n-2}{n-1}\left(\int_Xc_2(X)(\sigma\overline{\sigma})^{n-1}\right) \cdot q_X(L)^{n-1}.
\]
The second Todd class of $X$ is a positive multiple of $c_2(X)$, and if $L$ is an ample line bundle, then $q_X(L)>0$. Therefore $a_{n-1}$ is positive if and only if $\int_Xc_2(X)(\sigma\overline{\sigma})^{n-1}$ is positive.  Fixing a HK metric compatible with the symplectic structure, last quantity is a positive multiple of the $L^2$-norm of the Riemann curvature tensor (see \cite{Nieper03}), hence positive.
\end{proof}

In view of the previous Lemma we can speak of the Riemann-Roch polynomial for a deformation \emph{type}. This has been done for the two series of examples introduced by Beauville.

\begin{exmp}[\cite{EGM01}, Lemma 5.1]\label{exmp:K3}
Let $X$ be a HK of $K3^{[n]}$-type, then the Riemann-Roch polynomial is given by
\[
RR_X(t) = \binom{t/2 + n + 1}{n}.
\]
\end{exmp}

\begin{exmp}[\cite{Nieper03}, Lemma 5.2]\label{exmp:Kum}
Let $X$ be a HK of $\text{Kum}_n$-type, then the Hilbert polynomial takes the form
\[
RR_X(t) = (n+1)\binom{t/2 + n}{n}.
\]
\end{exmp}

We will say that the Riemann-Roch polynomial is of $K3^{[n]}$-type or $\text{Kum}_n$-type if it corresponds to one of the two examples above. Now we can state precisely the main result of this section.

\begin{thm}\label{thm:RR}
The Riemann-Roch polynomials for the deformation class of OG6 and OG10 are of $Kum_3$-type and $K3^{[5]}$-type respectively.
\end{thm}

The theorem will be proved in Propositions \ref{prop:OG10} and \ref{prop:OG6} below.

\section{Abelian fibered CY varieties}\label{sec:abelianfiberedCY}

Let $\pi:X\to B$ be a flat surjective morphism with connected fibers between projective normal complex varieties. Denote by $X_b$ the schematic fiber of $b\in B$. For the rest of this section we assume that:
\begin{itemize}
    \item $X$ has rational singularities and $\omega_X$ is trivial.
    \item Every smooth fiber $X_b$ is an abelian variety.
\end{itemize}

Denote by $\O_B(1)$ an ample line bundle on $B$ and let $F = \pi^*(\O_B(1))$ be the pullback. Let $L$ be a $\pi$-ample line bundle on $X$. Whenever $X_b$ is smooth the restriction $L_b := L|_{X_b}$ defines a polarization of the abelian variety $X_b$.

Recall that to any polarization on an abelian variety one can associate a tuple of positive integers $(d_1,\dots,d_n)$ which is called the polarization type, see \cite{Lange92}, in the following way:
Since $X_b$ is an abelian variety we have an identification $H^2(X_b,\Z)\cong \bigwedge^2H_1(X_b,\Z)^\vee$, hence we can interpret $L_b$ as an alternating integral form on the lattice $H_1(X_b,\Z)$. Therefore we can find a basis of $H_1(X_b,\Z)$ for which $L_b$ has the form
\[
\begin{pmatrix}
0 & D\\
-D & 0
\end{pmatrix}
\]
where $D = \mathrm{diag}(d_1,\dots,d_n)$ is an integral diagonal matrix with $d_i>0$ and $d_i|d_{i+1}$. We will denote by $(d_1,\dots,d_n)$ the type of $L_b$. Since the morphism is flat the type remains constant on the smooth locus of $\pi$. The following is a generalization of \cite[Claim 12]{Saw16}.

\begin{thm}\label{thm:sawon}
Let $L$ be a $\pi$-ample line bundle on $X$ and let $(d_1,\dots,d_n)$ be the type of $L_b$ for a smooth fiber $X_b$. Then for any $m\in\Z$ the sheaf $\pi_*(L\otimes F^{\otimes m})$ is locally free of rank $d_1\cdot\cdot\cdot d_n$ and all higher direct images vanish. Moreover, 
\[
h^p(X_b,(L\otimes F^{\otimes m})|_{X_b}) =  
\begin{cases}
d_1\cdot\cdot\cdot d_n & p=0, \\
0 & p>0.
\end{cases}
\]
\end{thm}
\begin{proof}
Let $k>0$ be an integer such that $M = L\otimes F^{\otimes k}$ is ample. Let $X_b$ be a smooth fiber, and denote by $M_b$ the restriction of $M$. Then
\[
h^p(X_{b},M_b) = h^p(X_b,L_b)=
\begin{cases}
d_1\cdot\cdot\cdot d_n & p=0,\\
0 & p>0.
\end{cases}
\]
Therefore the higher direct image sheaves $R^p\pi_*M$ are torsion for $p>0$. Let $\epsilon:\widetilde{X}\to X$ be a resolution of singularities of $X$. Since $X$ has rational singularities, hence $\epsilon_*(\omega_{\widetilde{X}}) =\omega_X=\O_X$ and $R^q\epsilon_*\omega_{\widetilde{X}} = 0$ for every $q>0$. Therefore the Grothendieck spectral sequence 
\[
R^p\pi_*(R^q\epsilon_*\omega_{\widetilde{X}}\otimes M)\implies R^{p+q}(\epsilon\circ\pi)_*(\omega_{\widetilde{X}}\otimes \epsilon^*M).
\]
degenerates and so $R^{p}(\epsilon\circ\pi)_*(\omega_{\widetilde{X}}\otimes \epsilon^*M)\cong R^p\pi_*(M)$. On the other hand, the divisor $\epsilon^*(M)$ is big and nef, so Theorem 2.2. in \cite{Hacon04} states  that
$R^{p}(\epsilon\circ\pi)_*(\omega_{\widetilde{X}}\otimes \epsilon^*M)$ is torsion-free for $p\geq 0$. We conclude that $R^p\pi_*(M)$ must vanish for $p>0$.

Theorem 12.11 of \cite{Hart77} states that if $H^p(X_b,M_b)$ vanishes for all $b\in B$, then the natural map
\[
R^{p-1}\pi_*M\otimes_{\O_b}k(b)\to H^{p-1}(X_b,M_b)
\]
is an isomorphism for all $b\in B$. Since $H^{n+1}(X_b, M_b)$ vanishes for all $b\in B$ by dimension reasons, and by the previous reasoning also $R^n\pi_*(M) = 0$, then $H^n(X_b,M_b)$ vanishes for all $b\in B$. Continuing by reverse induction, we find that $H^p(X_b,M_b)=0$ for all $p>0$ and all $b\in B$.

Finally, $\pi:X\to B$ is a flat family and 
\[
h^0(X_b,M_b) = \chi(X_b,M_b)
\]
is topological, so for all $b\in B$ we find that $h^0(X_b,L_b)$ agrees with the value $d_1\cdot\cdot\cdot d_n$ for a smooth fibre. Theorem 12.11 of \cite{Hart77} implies that the sheaf $\pi_*(M)$ is locally free of rank $d_1\cdot\cdot\cdot d_n$ as claimed. The projection formula for higher direct images implies
\[
R^p\pi_*(L\otimes F^{\otimes m}) \cong R^p\pi_*(L\otimes F^{\otimes k})\otimes \O_{B}(m-k) \cong R^p\pi_* M\otimes\O_{B}(m-k)
\]
for any $m\in\Z$, the proposition follows.
\end{proof}

Now that we have proved that the sheaf $\pi_*L$ is locally free we can study its positivity properties. The next Theorem follows closely Mourougane' strategy in \cite{Mourougane97}. We will need the following well-known Lemma. Recall that a vector bundle $\mathcal{E}$ on a projective variety is called \emph{nef} if the canonical line bundle $\O_{\P(\mathcal{E})}(1)$ is nef.
\begin{lem}\label{lem:ggthennef}
Let $\mathcal{E}$ be a vector bundle over a projective variety $Y$. If there exists an ample line bundle $M$ on $Y$ such that for all $s$ the vector bundle $\mathcal{E}^{\otimes s}\otimes M$ is globally generated, then $\mathcal{E}$ is nef.
\end{lem}
\begin{proof}
See Example 6.2.13 in \cite{PAG2}.
\end{proof}

\begin{thm}
Let $L$ be a big and nef line bundle on $X$ that is $\pi$-ample. Then the vector bundle $\pi_*L$ is nef.
\end{thm}
\begin{proof}
For any integer $s>0$ define $X^{(s)} := X\times_{B}\dots\times_{B} X$ and let $\pi^{(s)}:X^{(s)}\to B$ be the induced map. This is a flat map because flatness is preserved under base-change. Denote by $\text{pr}_i:X^{(s)}\to X$ the $i$-th projection. Define the line bundle $L^{(s)} := \otimes_{i=1}^s \text{pr}_i^*(L)$.

\begin{claim}
$\pi_*(L)^{\otimes s}\cong \pi^{(s)}_*(L)$.
\end{claim}
\begin{pfclaim}
Indeed, we will proceed by induction, the case $s=1$ being trivial. Use the following diagram 
\begin{equation}\label{eq:vieweghproducttrick}
\begin{tikzcd}
X^{(s)}\ar[rr,"\text{pr}_s"]\ar[d,"p"]\ar[rrd,"\pi^{(s)}"] & & X\ar[d,"\pi"]\\
X^{(s-1)}\ar[rr,"\pi^{(s-1)}"] & & B
\end{tikzcd}
\end{equation}
where $p$ denotes the canonical map given by base-change. Apply projection formula twice and flat base change
\[
\begin{split}
\pi_*^{(s)}(L^{(s)}) &= \pi_*^{(s)}(p^*(L^{(s-1)}\otimes \text{pr}_s^*(L)))\\
&= \pi_*^{(s-1)}(p_*(p^*(L^{(s-1)}\otimes \text{pr}_s^*(L)))) \\
&= \pi_*^{(s-1)}(L^{(s-1)}\otimes p_*(\text{pr}_s^*(L)))\\
&= \pi_*^{(s-1)}(L^{(s-1)}\otimes {\pi^{(s-1)}}^*(\pi_*(L))\\
&= \pi_*(L^{(s-1)})\otimes \pi_*(L)) = \pi_*(L)^{\otimes s}.
\end{split}
\]
Last equality follows by induction hypothesis.
\qed
\end{pfclaim}

We want to apply Lemma \ref{lem:ggthennef} to our case. By Theorem \ref{thm:sawon} we know that $R^p\pi_* L = 0$ for $p>0$ so the Leray Spectral Sequence for $L$ degenerates and we have that
\begin{equation}\label{eq:leray}
H^i(B,\pi_*(L\otimes F^{\otimes m})) = H^i(X,L\otimes F^{\otimes m}) = 0 
\end{equation}

for all $i>0$ and $m>0$. By replacing $\O_B(1)$ with a suitable multiple we can assume it to be very ample, hence $\pi_*L$ is $n$-regular with respect to $\O_B(1)$ in the sense of Castelnuovo-Mumford. Therefore the vector bundle $\pi_*L\otimes\O_{B}(n)$ is globally generated. This is the case $s=1$ of  Lemma \ref{lem:ggthennef}, for $s>1$ is just an application of Künneth formula. Indeed, for all $m>0$ we have
\[
H^p(X^{(s)},L^{(s)}\otimes {\pi^{(s)}}^*(\O_{B}(m))) = \bigoplus_{i+j=p} H^i(X^{(s-1)},L^{(s-1)})\otimes H^j(X,L\otimes F^{\otimes m})
\]
The line bundle $L$ is big and nef, so by Kawamata-Viehweg all of its higher cohomology groups vanish. We apply Künneth formula once again and use equation \eqref{eq:leray} to conclude that $H^p(X^{(s)},L^{(s)}\otimes {\pi^{(s)}}^*(\O_{B}(m)))$ must vanish for all $p>0$ and $m\geq 0$. This ensures that the sheaf ${\pi_{(s)}}_*(\O_X(L))$ is $n$-regular. Therefore 
\[
 {\pi_{(s)}}_*(L)\otimes\O_{B}(n) \cong \pi_*(\O_X(L))^{\otimes s}\otimes\O_{B}(n)
 \]
 is also globally generated. By Lemma \ref{lem:ggthennef} the vector bundle $\pi_*(L)$ is nef.
\end{proof}

\section{Riemann-Roch polynomial for OG10}\label{sec:RROG10}

We will use the realization of OG10 constructed by Laza-Saccà-Voisin in \cite{LSV17}, denoted by $J$. This has a Lagrangian fibration $\pi:J\to \P^5$. Let $\Theta$ be the relative theta divisor. This is $\pi$-ample by \cite{LSV17}, Section 5. Let $J_t$ be a smooth fiber of $\pi$ and let $\Theta_t$ be the restriction of $\Theta$ to the fiber. By Theorem \ref{thm:sawon} we have that $\pi_*\O_J(\Theta)$ is locally free of rank $h^0(J_t,\Theta_t)$. The relative theta divisor is a principal polarization when restricted to any smooth fiber, hence $h^0(J_t,\Theta_t) = 1$. We obtain

\begin{prop}
$\pi_*(\O_J(\Theta))\cong \O_{\P^5}(k)$ for some $k\geq 0$. In particular
\begin{equation}\label{eq:binomialRR}
   \chi(J,\Theta +mF) = h^0(J,\Theta +mF) = \binom{k+m+5}{5}
\end{equation}
for all $m\geq 0$.
\end{prop}
The first equality follows by Theorem \ref{thm:sawon}
The class $F$ is isotropic with respect to $q_J$, we use the polarized version of the Fujiki formula to get
\[
\int_J \Theta^5\smile F^5 = \frac{c_J}{10!}\cdot (5!)^2\cdot 2^5\cdot q_J(F,\Theta)^5.
\]

On the other hand, we can compute the left hand side to be

\[
\int_J \Theta^5\smile F^5  = \int_{J_t}\Theta^5|_{J_t} = \int_{J_t}\Theta_t^5 = 5!
\]
Hence $q_J(\Theta,F) = 1$. In particular $q_J(\Theta + mF) = q_J(\Theta) + 2m$.

\begin{prop}\label{prop:OG10}
The Riemann-Roch polynomial of OG10 is of $K3^{[5]}$-type.
\end{prop}
\begin{proof}
In \eqref{eq:binomialRR} we take $t = q(\Theta + mF) = q(\Theta) +2m$ to get
\[
RR_J(t) = \binom{k+\frac{t-q_J(\Theta)}{2} + 5}{5}.
\]
Evaluating at zero we get the following equation
\[
6 = \binom{k-\frac{q_J(\Theta)}{2} + 5}{5}
\]
whose only rational solution is $k-\frac{q_J(\Theta)}{2} =1$, the result follows from Example \ref{exmp:K3}.
\end{proof}

We conclude this section by showing that the same strategy can be used to compute the Riemann-Roch polynomial for HK manifolds of $\mathrm{K}3^{[n]}$-type. Let $(S,L)$ be a K3 surface with $\mathrm{Pic}(S)$ generated by $L$ and let $L^2 = 2d$. By Riemann-Roch we have that $|L|\cong\mathbb{P}^{d+1}$ and every smooth curve $C$ in the linear system is of genus $d+1$. Let $\mathcal{C}/\mathbb{P}^{d+1}$ be the universal family of all curves linearly equivalent to $C$. By the hypothesis on the Picard group of $S$ we have that every curve in the linear system is reduced and irreducible, and therefore its compactified Jacobian is well-defined as the moduli space of rank-one torsion-free sheaves on the curve of degree $k$. Thus we get a fibration $\overline{\mathrm{Jac}}^k(\mathcal{C}/\mathbb{P}^{d+1})\rightarrow\mathbb{P}^{d+1}$ whose general fibre is a $d+1$-dimensional abelian variety.

On the other hand, the moduli space $M(0,L,k-d)$ of $L$-stable sheaves on $S$ with Mukai vector $(0,L,k-d)$ is a HK variety of dimension $2d+2$ by \cite[Example 0.5]{Mukai1984}. The general element of this moduli space is again a degree $k$ line bundle on a smooth curve in the linear system $|L|$, thought of as a torsion sheaf on $S$. We can also think of the fibration structure by considering the map
\begin{equation}\label{eq:supportmorphism}
\mathrm{Supp}: M(0,L,k-d)\longrightarrow \P^{d+1}
\end{equation}
taking a sheaf $\F$ to its support $\mathrm{Supp}\F\in|L|$. With the hypothesis on $S$ these two spaces are isomorphic. The following theorem is well-known. See, for example Section 3 in \cite{FMOS2022} for details.

\begin{thm}
Let $(S,L)$ be a K3 surface with $\mathrm{Pic}(S)\cong\Z L$, then all the fibers of the morphism defined in \eqref{eq:supportmorphism} are irreducible. Moreover, the space $M(0,L,0)$ admits a canonically defined \emph{theta divisor} $\Theta_L$ given by
\[
\Theta_L = \{ \F\in M(0,L,0) \text{ such that } h^0(S,\F) = h^1(S,\F) \neq 0 \}
\]
which is ample on each fiber and with $q(\Theta_L) = -2$.
\end{thm}

The same proof as in the OG10 deformation type computes the Riemann-Roch polynomial, and since the dimension of $M(0,L,0)$ is $2d+2$, we cover all the dimensions of manifolds of $\mathrm{K}3^{[n]}$-type.

\section{Riemann-Roch polynomial for OG6}\label{sec:RROG6}

Although there does exist a HK manifold of OG6-type with a Lagrangian fibration, we cannot use the same strategy as we did for OG10 because the abelian varieties appearing as smooth fibers are not principally polarized --i.e. there is not an ample divisor restricting to a principal polarization on every smooth fiber. Hence, we use an alternative method based on the explicit description of some divisors. 

Let $X$ be a HK manifold of OG6-type. The formula \eqref{eq:chi} for $X$ is:

\begin{equation}\label{eq:RRforOG6}
    \chi(X,L) = a_0 + \frac{a_1}{2!}q_X(L) + \frac{a_2}{4!}q_X(L)^2 + \frac{a_3}{6!}q_X(L)^3 = RR_X(q_X(L))
\end{equation}

By Lemma \ref{lem:RRproperties} we have $a_0 = 4$ and also $a_3 = c_X = 60$, by \cite{Rap07}. We will find divisors whose invariants reduce equation \ref{eq:RRforOG6} to a linear system of equations. In order to do this, we will introduce the divisors $\widetilde{\Sigma}$ and $\widetilde{B}$. They are both effective divisors for the variety of OG6-type considered in \cite{Rap07}. Their Beauville-Bogomolov forms were already computed in Theorems 3.3.1 and 3.5.1 of  \cite{Rap07} and the Euler characteristic will be computed here.

\begin{table}[ht]
\centering
\begin{tabular}{l|l|l|}
\cline{2-3}
                                           & \textbf{$\chi$} & \textbf{$q_X$} \\ \hline
\multicolumn{1}{|l|}{$\widetilde{\Sigma}$} &        -4         & -8                \\ \hline
\multicolumn{1}{|l|}{$\widetilde{B}$}      &         0        & -2                \\ \hline
\end{tabular}\caption{Invariants for $\widetilde{\Sigma}$ and $\widetilde{B}$.}\label{table}
\end{table}

In the rest of this section we are going to compute the Euler characteristics in Table \ref{table}. For this we need explicit descriptions of the divisors above. Let $A$ be an abelian surface, and denote by $A^\vee$ its dual. Multiplication by $-1$ is an involution on both abelian surfaces and their product $A\times A^\vee$.

\begin{thm}[\cite{MRS2018}, Corollary 2.8]\label{thm:Sigma}
The divisor $\widetilde{\Sigma}$ is a $\P^1$-bundle over the nonsingular variety
\[
\overline{\Sigma} := Bl_{\text{Sing}((A\times A^\vee)/\pm1)}(A\times A^\vee)/\pm1.
\]
\end{thm}
\begin{thm}[\cite{Nagai2014}]\label{thm:Nagai}
$\widetilde{B}$ is a $\P^1$-bundle over $K(A)\times K(A^\vee)$, where $K(A)$ denotes the smooth Kummer surface associated to $A$.
\end{thm}

\begin{prop}\label{prop:eulerchar}
With the notation as above we have:
\[
\chi(\overline{\Sigma},\O_{\overline{\Sigma}}) = -8 \quad\text{and}\quad \chi(\widetilde{B},\O_{\widetilde{B}}) = 4.
\]
\end{prop}
\begin{proof}
The Euler characteristic is multiplicative for smooth fibrations, so $\chi(\O_{\widetilde{\Sigma}}) = \chi(\O_{\overline{\Sigma}})$. The Hodge numbers are given by
\[
H^{4,q}(\overline{\Sigma}) = H^{4,q}(A\times \widehat{A})^{inv}.
\]
Since the action in cohomology is given on differential forms by multiplication by $(-1)^{\deg}$, the cohomology groups are
\[
H^{4,q}(A\times \widehat{A}) = 
\begin{cases}
    0 & \text{if} \;\;\;\; q=2,4\\
    6 & \text{if} \;\;\;\; q=3\\
    1 & \text{if} \;\;\;\; q=1,5.
\end{cases}
\]
It follows that $\chi(\O_{\overline{\Sigma}}) = -8$. For the other divisor $\widetilde{B}$ in the table \ref{table} the Theorem \ref{thm:Nagai} yields $\chi(\O_{\widetilde{B}}) = \chi(\O_{K(A)\times K(A^\vee)}) = 4$, since both $K(A)$ and $K(A^\vee)$ are K3 surfaces. 
\end{proof}

The computations of the Euler characteristic given above compute the Euler characteristics for these divisors on their respective HK varieties by the following.

\begin{lem}\label{lem:eulercharacteristicdivisor}
Let $X$ be a HK variety of OG6-type and let $E\subseteq X$ be an effective smooth divisor, then $\chi(X,\O(E)) = 4-\chi(E,\O_E)$.
\end{lem}
\begin{proof}
The exact sequence induced by a nonzero section of $E$ gives a formula for its Euler characteristic:
\begin{equation}\label{eq:eulercharacteristicdivisor}
    \chi(X,\O(E)) = \chi(X,\O_X) + \chi(E,\O_E(E)) = 4 + \chi(E,\O_E(E)).
\end{equation}
By adjunction $\O_E(E)\cong\omega_E$. If $E$ is smooth, then by Serre duality $\chi(E,\omega_E) = -\chi(E,\O_E)$.
\end{proof}

Hence by Lemma \ref{lem:eulercharacteristicdivisor} and Proposition \ref{prop:eulerchar} we have computed the Euler characteristics in Table \ref{table}. The following proposition finishes the proof of the main theorem.

\begin{prop}\label{prop:OG6}
The Riemann-Roch polynomial of OG6 is of $Kum_3$-type.
\end{prop}
\begin{proof}
If we plug into Equation \eqref{eq:RRforOG6} the known invariants of $\widetilde{\Sigma}$ and $\widetilde{B}$ given in Table \ref{table}, then we have a complete system of \emph{linear} equations because the Beauville-Bogomolov square of both divisors is different. Hence the constants $a_1$ and $a_2$ are uniquely determined. The coefficients $a_0,\dots,a_4$ are the same as those of the Riemann-Roch polynomial of $Kum_3$-type given in Example \ref{exmp:Kum}.
\end{proof}

\bibliography{bib.bib}{}
\bibliographystyle{alpha}

\newpage
\appendix

\section{Chern numbers of OG10, by Yalong Cao and Chen Jiang}
An interesting question is to compute topological invariants (e.g.  topological Euler characteristic and Chern numbers) of hyperk\"ahler varieties. 

For a HK of $K3^{[n]}$-type, Ellingsrud--G\"{o}ttsche--Lehn \cite{egl} showed that the Chern numbers can be efficiently calculated,  in terms of the Chern numbers of the varieties $(\mathbb{P}^2)^{[k]}$ and $
(\mathbb{P}^1 \times \mathbb{P}^{1})^{[k]}$
(which  can be calculated by Bott's residue formula via \cite{ES87, ES96}), though no explicit formula is known.

For a HK of $\text{Kum}_{n}$-type, Nieper \cite{nieper2002} showed that the Chern numbers can be efficiently calculated,   in terms of the Chern numbers of the varieties $(\mathbb{P}^2)^{[k]}$ (see also \cite{BN, SawonPhd}).

For OG6, the Chern numbers are computed by Mongardi--Rapagnetta--Sacc\`{a}
\cite[Corollary 6.8]{mrs} (see also \cite{Rapa-thesis}).

For OG10, the topological Euler characteristic is computed by Mozgovyy \cite{Moz-thesis} (see also \cite{HLS}) and the Hodge numbers are computed by  de Cataldo--Rapagnetta--Sacc\`a \cite{CRS}. But the Chern numbers of OG10 have not been computed yet. In fact, Hodge numbers could not provide enough linear equations to solve Chern numbers (cf. \cite[Section 5.2]{SawonPhd}). 

As an application to the Riemann--Roch polynomial of OG10, combining with the Hodge numbers, we can compute $7$ Chern numbers of OG10.

\subsection{Relations of $q_X$ and $\lambda$ via Riemann--Roch polynomials}
Let $X$ be a HK variety. For a line bundle $L$ on $X$, Nieper \cite[Definition 17]{nieper} defined the {\it characteristic value} of $L$,
$$
\lambda(L):=\begin{cases}\frac{24n\int_{X}\ch(L)}{\int_{X}c_{2}(X) \ch(L)} & \text{if well-defined;}\\ 0 & \text{otherwise.}\end{cases}
$$
Note that $\lambda(L)$ is a positive (topological constant) multiple of $q_X(c_1(L))$ (cf. \cite[Proposition 10]{nieper}), more precisely,
$$
\lambda(L)=\frac{12c_X}{(2n-1)C(c_2(X))}q_X(c_1(L)).
$$
Here we denote the ratio $$M_{\lambda, q_X}:=\lambda(L)/q_X(c_1(L))=\frac{12c_X}{(2n-1)C(c_2(X))}.$$
So for the Riemann--Roch polynomial, we can write it in terms of either $q_X$ (say $\text{\rm RR}_{X, q_X}(t)$) or $\lambda$ (say $\text{\rm RR}_{X, \lambda}(t)$). They coincide up to a multiple of invariable, that is, $\text{\rm RR}_{X, q_X}(t)=\text{\rm RR}_{X, \lambda}(M_{\lambda, q_X}t)$.

 An observation (maybe well-known to experts) is that, even though we do not know the value of  $M_{\lambda, q_X}=\frac{12c_X}{(2n-1)C(c_2(X))}$ priorly, we can see it at once we know the expression of  the Riemann--Roch polynomial, that is, $\text{\rm RR}_{X, q_X}(t)$  determines $M_{\lambda, q_X}$,  $(\text{\rm RR}_{X, \lambda}(t), c_X)$  determines $M_{\lambda, q_X}$.
 Hence as a consequence, 
$\text{\rm RR}_{X, q_X}(t)$ and $(\text{\rm RR}_{X, \lambda}(t), c_X)$ determine each other.

\begin{lem}
The  coefficients of the first 2 leading terms of $\text{\rm RR}_{X, q_X}(t)$ determines $M_{\lambda, q_X}$. The  coefficient of the leading term $\text{\rm RR}_{X, \lambda}(t)$ and  $c_X$ determines $M_{\lambda, q_X}$.
\end{lem}

\begin{proof}
Write $\text{\rm RR}_{X, q_X}(t)=at^n+bt^{n-1}+(\text{\rm lower terms}).$ 
Then we have
$$
\int_X \td(X)\exp(L)=\text{\rm RR}_{X, q_X}(q_X(L)).
$$
Comparing both sides, 
this implies that 
$$
\int_X \frac{1}{(2n)!}L^{2n}=aq_X(L)^n,
$$
and 
$$
\int_X\td_2(X) \frac{1}{(2n-2)!}L^{2n-2}=bq_X(L)^{n-1}.
$$
Using Fujiki's relation, these implies that 
$$
c_X=(2n)!a
$$
and 
$$
C(c_2(X))=12(2n-2)!b.
$$
Hence $M_{\lambda, q_X}=\frac{12c_X}{(2n-1)C(c_2(X))}=\frac{2na}{b}$.

Write $\text{\rm RR}_{X, \lambda}(t)=a't^n+(\text{\rm lower terms}).$ 
Then
$\text{\rm RR}_{X, q_X}(t)=\text{\rm RR}_{X, \lambda}(M_{\lambda, q_X}t)=a'M_{\lambda, q_X}^nt^n+(\text{\rm lower terms})$. Then $(2n)!a'M_{\lambda, q_X}^n=c_X$, which determines $M_{\lambda, q_X}$ as $M_{\lambda, q_X}$ is a positive real number. 
\end{proof}

\subsection{Chern numbers of OG10}

\begin{thm}
Let $X$ be a HK of OG10-type. Then the Chern numbers of $X$ are the following:
\begin{align*}
&(c_2^5, c_2^3c_4, c_2^2c_6, c_2c_8, c_2c_4^2, c_4c_6, c_{10})\\
=&(127370880, 53071200, 12383280, 1791720, 22113000, 5159700, 176904)
\end{align*}
\end{thm}
\begin{proof}
Let $X$ be a HK of OG10-type.
Then the Riemann--Roch polynomial is
$$
\text{\rm RR}_{X, q_X}(t)=\binom{t/2+6}{5}.
$$
By the discussion of the first section, 
we know that 
\begin{equation}
\text{\rm RR}_{X, \lambda}(t)=\text{\rm RR}_{X, q_X}(4t)=\binom{2t+6}{5}.
\label{ortiz RR} \end{equation}
That is, for a line bundle $L$ on $X$, 
$$
\chi(X,L)=\binom{2\lambda(L)+6}{5}.
$$
On the other hand, Nieper \cite[Theorem 5.2]{nieper} proved that
\begin{align}\chi(X,L){}&=\int_{X}\exp\bigg(-2\sum_{k=1}^{\infty}\frac{B_{2k}}{4k}\ch_{2k}(X)T_{2k}\bigg(\sqrt{\frac{\lambda(L)}{4}+1}\bigg)\bigg)\notag\\
{}&=\int_{X}\exp\bigg(-2\sum_{k=1}^{\infty}b_{2k}s_{2k}(X)T_{2k}\bigg(\sqrt{\frac{\lambda(L)}{4}+1}\bigg)\bigg)
\label{nieper RR} \end{align}
where $B_{2k}$ are the Bernoulli numbers, 
$b_{2k}=\frac{B_{2k}}{4k(2k)!}$ are modified Bernoulli numbers with $b_0=1$, $s_{2k}=(2k)!\ch_{2k}(X)$, and $T_{2k}$ are even Chebyshev polynomials of the first kind.
So by \eqref{ortiz RR} and \eqref{nieper RR}, we get an identity between polynomials in terms of $y=\sqrt{{\lambda(L)}/{4}+1}$:
\begin{equation}
\binom{8y^2-2}{5}=\int_{X}\exp\bigg(-2\sum_{k=1}^{\infty}b_{2k}s_{2k}(X)T_{2k}(y)\bigg).\label{ortizRR=nieperRR}
\end{equation}
Comparing coefficients of \eqref{ortizRR=nieperRR}, we have $6$ linear equations of $7$ Chern numbers $c_2^5$, $c_2^3c_4$, $c_2^2c_6$, $c_2c_8$, $c_2c_4^2$, $c_4c_6$, $c_{10}$.
Of course this is not sufficient to get a unique solution. In fact, among $6$ equations, there are only $3$ linearly independent linear equations (from comparing coefficients of $y^8$, $y^6$, and $y^2$).

On the other hand, we have more equations form the Hodge numbers of $X$ computed by de Cataldo--Rapagnetta--Sacc\`a \cite{CRS}.
This gives us 
\begin{align*}
\chi^{1}(X){}&=-111;\\
\chi^{2}(X){}&=1062;\\
\chi^{3}(X){}&=-7151;\\
\chi^{4}(X){}&=33534.
\end{align*}
Here $\chi^p(X)=\sum_{q=0}^{\dim X}(-1)^{q}h^{p,q}(X)=\int_X\ch(\Omega_X^p)\td(X)$.
Again, expressing the left-hand sides by Chern numbers, we get $4$  linear equations  of $7$ Chern numbers. 
We can use Mathematica to solve the linear equations for 
$7$ Chern numbers of OG10 as the following:
\begin{align*}
&(c_2^5, c_2^3c_4, c_2^2c_6, c_2c_8, c_2c_4^2, c_4c_6, c_{10})\\
=&(127370880, 53071200, 12383280, 1791720, 22113000, 5159700, 176904).
\end{align*}
\end{proof}

\end{document}